\documentclass[a4,12pt]{article}%
\usepackage{amsmath}
\usepackage{amsfonts}
\usepackage{amssymb}
\usepackage{enumerate}
\usepackage{amsthm}
\usepackage{thmtools}
\usepackage{blkarray}

\newtheorem{theorem}{Theorem}[section]

\newtheorem{definition}[theorem]{Definition}
\newtheorem{example}[theorem]{Example}

\newtheorem{lemma}[theorem]{Lemma}

\newtheorem{comment}[theorem]{Remark}

\newcommand{\calM}{{\cal M}}
\newcommand{\I}{{\rm I}}

\newcommand{\E}{{\mathbb E}}

\newcommand{\len}{{\rm stage}}

\newcommand{\prob}{{\mathbb P}}

\newcommand{\val}{{\rm val}}
\newcommand{\Val}{{\rm Val}}

\newcommand{\dN}{{\mathbb N}}

\newcommand{\dR}{{\mathbb R}}

\newcommand{\calF}{{\cal F}}

\newcommand{\ep}{\varepsilon}

\title{Equilibria in Repeated Games with Countably Many Players and Tail-Measurable Payoffs%
\thanks{We would like to thank Abraham Neyman for his valuable comments. Ashkenazi-Golan and Solan
acknowledge the support of the Israel Science Foundation, grants
\#217/17 and \#722/18, and the NSFC-ISF Grant \#2510/17. This work has been partly supported by COST Action CA16228 European Network for Game Theory.}}

\author{Galit Ashkenazi-Golan,%
\thanks{School of Mathematical Sciences, Tel-Aviv University, Tel-Aviv, Israel, 6997800}
J\'{a}nos Flesch,%
\thanks{Department of Quantitative Economics,
Maastricht University, P.O.Box 616, 6200 MD, The Netherlands}
Arkadi Predtetchinski,%
\thanks{Department of Economics, Maastricht University, P.O.Box 616, 6200 MD,
The Netherlands}
Eilon Solan%
\thanks{School of Mathematical Sciences, Tel-Aviv University, Tel-Aviv, Israel, 6997800}}

\begin{document}

\maketitle

\noindent{\textbf Keywords:} repeated games, Nash equilibrium, countably many players, tail-measurable payoffs

\begin{abstract}
We prove that every repeated game with countably many players, finite action sets,
and tail-measurable payoffs admits an $\ep$-equilibrium,
for every $\ep >  0$.
\end{abstract}

\section{Introduction}

Nash equilibrium (Nash \cite{nash1950equilibrium}) is the most prevalent and important solution concept in game theory to date,
and is applied to a multitude of problems in, e.g., economic theory, computer science, political science, and life sciences (Holt and Roth
\cite{holt2004nash}).
It exists in one-shot normal-form games (games where the players choose their actions simultaneously) with finitely many players,
compact action sets, and bounded and jointly continuous payoff functions (Glicksberg \cite{glicksberg1952further}) and in extensive form finite games with perfect information, where the players move alternatingly (Kuhn \cite{kuhn1950extensive}). 
Gale \cite{gale1953theory} provided an algorithm to compute one such equilibrium in extensive form finite games with perfect information.

The conditions under which Nash equilibrium exists are tight:
two-player zero-sum games where the payoff function is not continuous may fail to have a Nash equilibrium,
and even an $\ep$-equilibrium for $\ep > 0$ sufficiently small,
like the choose-the-largest-integer game;
and the same holds for games with countably many players and finite action sets
(see Peleg \cite{peleg1969equilibrium}, Voorneveld \cite{voorneveld2010possibility}, and Section~\ref{sec2}.\ref{subsecn.example} below).

When the interaction is repeated,
the payoff depends on the whole play,
and new phenomena arise, 
like learning, reciprocity, and cooperation. 
These phenomena were studied in various disciplines,
like 
economics (e.g., Fudenberg and Maskin \cite{fudenberg1986the} and Mailath and Samuelson \cite{mailath2006repeated}),
mathematics (e.g., Aumann and Maschler \cite{aumann1995repeated}, and
Martin \cite{martin1975borel}, \cite{martin1998determinacy}), 
and evolution (e.g., Press and Dyson \cite{press2012iterated}, Stewart and  Plotkin (\cite{stewart2012extortion}, \cite{stewart2013extortion}), Hilbe et. al. \cite{hilbe2013evolution}, Hilbe et. al. \cite{hilbe2014cooperation}, or McAvoy and Hauert \cite{mcavoy2016autocratic}).
The concept of Nash equilibrium has been studied in more general classes of infinite-stage games, like stochastic games
(e.g., Shapley \cite{shapley1953stochastic}, Solan and Vieille \cite{solan2015stochastic}), 
differential games (e.g., Isaacs \cite{isaacs1999differential}),
and mean field games (e.g., Lasry and Lyons \cite{lasry2007mean} and Huang, Malham\'e and Caines \cite{huang2006large}).

In a game that is infinitely repeated, whether an $\ep$-equilibrium exists is  known only in few cases,
such as when the payoff is a continuous function of the play,
%like the so called \emph{compact case} (which includes discounted games), where for every $\ep > 0$ there is $t \in \dN$
%such that up to $\ep$ the payoff is determined by the play in the first $t$ stages,
or when the payoff is the limit of the averages of stage payoffs.

In this paper we extend the equilibrium existence result in two ways.
First, we allow the payoff function to be tail-measurable.
Second, we allow the set of players to be countable
(for models with countably many players, see, e.g., \cite{bala1998learning} or \cite{acemoglu2011bayesian}).
Thus, interestingly, even though one-shot games with countably many players and finite action sets may fail to have an $\ep$-equilibrium,
repeated games with countably many players, finite action sets, and tail-measurable payoffs do have an $\ep$-equilibrium.

To prove the result we introduce a new technique to the study of repeated games with general payoff functions,
which combines Martin's method \cite{martin1998determinacy} for Borel games and Blackwell games
with a method developed by Solan and Vieille \cite{solan2002correlated} for stochastic games.

The structure of the paper is as follows. In Section \ref{sec2}, we introduce the model and the existence result, and discuss an illustrative example. In Section \ref{sec3}, we prove the main result.
In Section~\ref{secn.extension} we discuss extensions.
In Section~\ref{sec4} we provide concluding remarks.

\section{Model and Main Result}\label{sec2}

\subsection{The Model}

\begin{definition}\label{def-game}
A \emph{repeated game with countably many players} is a triplet
$\Gamma = (I,(A_i)_{i \in I}, (f_i)_{i \in I})$, where
\begin{itemize}
\item   $I = \dN = \{0,1,2,\ldots\}$ is a countable set of players.
\item   $A_i$ is a nonempty and finite action set for player~$i$, for each $i \in I$. Let $A = \prod_{i \in I} A_i$ denote the set of action profiles. 
\item   $f_i : A^\dN \to \dR$ is player~$i$'s payoff function, for each $i \in I$.
\end{itemize}
\end{definition}

The game is played in discrete time as follows. In each stage $t \in \dN$,
each player $i \in I$ selects an action $a_i^t \in A_i$, independently of the other players. This induces an action profile $a^t=(a_i^t)_{i\in I}$ in stage $t$, which is observed by all players. Given the realized play $p = (a^0,a^1,\ldots)$, the payoff to each player~$i$ is $f_i(p)$.\medskip

\noindent\textbf{Histories and plays.} A \emph{history} is a finite sequence of elements of $A$, including the empty sequence. The set of histories is denoted by $H = \bigcup_{t=0}^\infty A^t$.
The current stage (or length) of a history $h\in A^t$ is denoted  by $\len(h)=t$.

A \emph{play} is an infinite sequence of elements of $A$. The set of plays is thus $A^\dN$. For a play $p \in A^\dN$ write $h^t(p)$ to denote the prefix of $p$ of length $t$. 
\medskip

\noindent \textbf{Measure theoretic setup.} We endow the set $A$ of action profiles, the sets $A^t$ for $t\geq 2$, and the set $A^\dN$ of plays with the product topology, where the action set $A_i$ of each player $i\in I$ has its natural discrete topology. We denote the corresponding Borel sigma-algebras by $\mathcal{F} (A)$, $\mathcal{F} (A^t)$, and $\mathcal{F} (A^\dN)$. These coincide with the respective product sigma-algebras, given the discrete sigma-algebra on each $A_i$. A payoff function $f_i$ is called \emph{product-measurable} if it is measurable w.r.t.~$\mathcal{F} (A^\dN)$.
\medskip

\noindent\textbf{Strategies.}
A \emph{mixed action} for player $i$ is a probability measure $x_i$ on $A_i$.
The set of mixed actions for player $i$ is denoted by $X_i$.
A \emph{mixed action profile} is a vector of mixed actions $x = (x_i)_{i \in I}$, one for each player.
The set of mixed action profiles is denoted by $X$.

A (behavior) \emph{strategy} of player~$i$ is a function $\sigma_i : H \to X_i$ that satisfies the following measurability condition: for each stage $t\in\dN$ and each action $a_i\in A_i$, the mapping from $A^t$ to $[0,1]$ given by $h\in A^t\mapsto\sigma_i(h)(a_i)$ is measurable. We denote by $\Sigma_i$ the set of strategies of player~$i$.

A \emph{strategy profile} is a vector of strategies $\sigma = (\sigma_i)_{i \in I}$, one for each player. We denote by $\Sigma$ the set of strategy profiles.\medskip

\noindent\textbf{Strategy profiles with finite support.}
A mixed action profile $x$ is said to have \emph{finite support} if only finitely many players randomize in $x$, i.e., all players except finitely many play a pure action with probability 1. We denote by $X^F$ the set of mixed action profiles with finite support.

A strategy profile $\sigma$ is said to have \emph{finite support} if $\sigma(h) \in X^F$ for every history $h\in H$.
The set of players who randomize may be history dependent,
and its size may not be bounded.
We denote by $\Sigma^F$ the set of strategy profiles with finite support.\medskip

\noindent\textbf{Definitions for the opponents of a player.}
For $i\in I$ let $-i$ denote the set of its opponents, $I \setminus \{i\}$. We can similarly define mixed action profiles $x_{-i} = (x_j)_{j\neq i}$, strategy profiles $\sigma_{-i} = (\sigma_j)_{j\neq i}$, mixed action profiles with finite support, and strategy profiles with finite support. The corresponding sets are denoted by $X_{-i}$, $\Sigma_{-i}$, $X_{-i}^F$, and $\Sigma_{-i}^F$, respectively.\medskip

\noindent\textbf{Expected payoffs.} For each history $h\in H$ the mixed action profile $\sigma(h)=(\sigma_i(h))_{i\in I}$ induces a unique probability measure $\mathbb{P}_{\sigma(h)}$ on $(A,\mathcal{F} (A))$ by Kolmogorov's extension theorem, and then $\sigma$ induces a unique probability measure $\mathbb{P}_\sigma$ on $(A^\dN,\mathcal{F} (A^\dN))$ by the Ionescu-Tulcea theorem. If player $i$'s payoff function $f_i$ is bounded and product-measurable, then
player $i$'s expected payoff under $\sigma$ is 
\[\E_{\sigma}[f_i]\,=\,\int_{A^\dN}f_i(p)\ \mathbb{P}_\sigma(dp).\]
In that case, player $i$'s expected payoff under $\sigma$ in the subgame starting at a history $h$, denoted by $\E_{\sigma}[f_i| h]$, is defined similarly.\medskip

\noindent\textbf{Equilibrium.} Consider a game with countably many players and bounded and product-measurable payoffs. Let $\vec \ep=(\ep_i)_{i\in I}$, where $\ep_i\geq 0$ for each player $i\in I$. A strategy profile $\sigma^*$ is called an \emph{$\vec \ep$-equilibrium}, if for every player $i \in I$ and every strategy $\sigma_i \in \Sigma_i$,
we have $\E_{\sigma^*}[f_i] \,\geq\, \E_{(\sigma_i,\sigma^*_{-i})}[f_i] - \ep_i$. 

\subsection{Tail Measurability}\label{Sec.Tail}

A set $B \subseteq A^\dN$ of plays is called \emph{tail} if whenever $p = (a^0,a^1,a^2,\dots) \in B$
and $p' = (a'^0,a'^1,a'^2,\dots)$ satisfies $a'^t = a^t$ for every $t$ sufficiently large, we have $p'\in B$. Intuitively,
a set $B \subseteq A^\dN$ is tail if changing finitely many coordinates does not change the membership relation for $B$.
The tail sets of $A^\dN$ form a sigma-algebra, denoted by $\calF^{Tail}$.
The sigma-algebras $\calF^{Tail}$ and $\calF(A^\dN)$ are not related: neither of them includes the other (as soon as infinitely many players have at least two actions); see Rosenthal \cite{rosenthal1975nonmeasurable} and Blackwell and Diaconis \cite{blackwell1996non}. A function from $A^\dN$ to $\dR$ is called \emph{tail-measurable} if it is measurable with respect to $\calF^{Tail}$.

Many evaluation functions in the literature of dynamic games are tail-measurable.
For example,
the lim\,sup of the average stage payoffs
(see, e.g., Mertens and Neyman \cite{mertens1981stochastic,mertens1982stochastic})
and the lim\,sup of stage payoffs
(see, e.g., Maitra and Sudderth \cite{maitra1993borel})
are both tail-measurable.
The discounted payoff
(see, e.g., Shapley \cite{shapley1953stochastic})
is not tail-measurable.

Various classical winning conditions that are studied in the computer science literature are also tail-measurable, even though they are usually expressed in terms of a state space, such as the B\"uchi, co-B\"uchi, parity, Streett, and M\"uller winning conditions (see, e.g., Horn and Gimbert \cite{Horn2008optimal}, and Chatterjee and Henzinger \cite{chatterjee2012survey}).

An iterative way to construct tail-measurable functions from $A^\dN$ to $\dR$ is the following:
If $(f_k)_{k \in \dN}$ is a sequence of tail-measurable functions from $A^{\dN}$ to $\dR$,
and if $\phi : \dR^\infty \to \dR$ is a function,
then the function $f(p) = \phi(f_0(p),f_1(p), f_2(p),\ldots)$ is tail-measurable.

A closely related property is shift-invariance, cf.~Chatterjee \cite{chatterjee2007concurrent}.
A set $B \subseteq A^\dN$ of plays is called shift-invariant
if whenever $(a^0,a^1,a^2,\ldots) \in B$ we have $(a^1,a^2,a^3,\ldots)\in B$ and $(a',a^0,a^1,a^2,\ldots)\in B$
for every $a'\in A$.
%The collection of all shift-invariant sets is a sigma-algebra. 
As one can verify, every shift-invariant set is tail, and thus every shift-invariant function from $A^\dN$ to $\dR$ is also tail-measurable.
The converse is not true: For $A=\{0,1\}$, consider the set $B$ of those sequences in $A^\dN$ that, at even coordinates,
contain only finitely many 1's  (with no restriction at odd coordinates). The set $B$ is tail but not shift invariant.

\subsection{The Main Result}

We now present the main result of the paper.

\begin{theorem}
\label{theorem:1}
Consider a repeated game with countably many players. If each player's payoff function is bounded, product-measurable, and tail-measurable, then the game admits an $\vec \ep$-equilibrium for each $\vec \ep=(\ep_i)_{i\in I}$, where $\ep_i>0$ for each $i\in I$.
\end{theorem}

We show the existence of an $\vec \ep$-equilibrium with a simple structure: The players are supposed to follow a play, that is, each player is supposed to play a fixed sequence of actions. If a player deviates, her opponents switch to a strategy profile with finite support to punish the deviator. 
Thus, it never happens that, on or off the equilibrium play, infinitely many players randomize at the same stage. 

%According to our definition of a repeated game,
%the action set of each player is fixed throughout the game.
%The result and the proof of Theorem \ref{theorem:1} do not change when the action sets %of the players depend on the stage, as long as they are finite at each stage.

A 0-equilibrium does not always exist under the condition of Theorem \ref{theorem:1}, not even in games with only one player (the other players can be treated as dummy). Indeed, suppose that this player has two actions 0 and 1. For each play $p=(a^0,a^1,\ldots)$, let $\zeta(p)$ denote the limsup-frequency of action 1: $\zeta(p)=\limsup_{t\to\infty}(a^0+\cdots+a^t)/(t+1)$. Define the payoff to be $\zeta(p)$ if $\zeta(p)<1$ and $0$ if $\zeta(p)=1$. This payoff function is tail-measurable, but the player cannot obtain payoff 1.

The example of the next section illustrates the main ideas of the construction of $\vec \ep$-equilibria for Theorem \ref{theorem:1}.

\subsection{An illustrative example}\label{subsecn.example}

Voorneveld \cite{voorneveld2010possibility} studied the following one-shot game with countably many players, where each player would like to choose the action that is chosen by the minority of the players:
\begin{itemize}
\item   The set of players is $I = \dN$.
\item   The action set of each player $i \in I$ is $A_i = \{0,1\}$.
\item   The payoff of each player~$i\in I$ is 
\begin{equation}
\label{equ:1}
r_i(a) = \left\{
\begin{array}{ll}
1, & \ \ \hbox{if } a_i = 0 \hbox{ and } \limsup\limits_{n \to \infty} \frac{1}{n}\sum_{j<n} a_j > \tfrac{1}{2},\\[0.3cm]
1, & \ \ \hbox{if } a_i = 1 \hbox{ and } \limsup\limits_{n \to \infty} \tfrac{1}{n}\sum_{j<n} a_j \leq \tfrac{1}{2},\\[0.1cm]
0, & \ \ \hbox{otherwise.}
\end{array}
\right.
\end{equation}
\end{itemize}

Voorneveld \cite{voorneveld2010possibility} showed that this one-shot game has no $\vec\ep$-equilibrium, provided $\ep_i\in[0,1/2)$ for every $i$.
Indeed, suppose by way of contradiction that $\sigma^*$ is an $\ep$-equilibrium.
The set of action profiles
\[B = \Big\{a\in A\colon \limsup_{n \to \infty} \tfrac{1}{n}\sum_{j<n} a_j > \tfrac{1}{2}\Big\} \subseteq \prod_{i\in I} A_i\]
is a tail set. Hence, by Kolmogorov's 0-1 law, either $\prob_{\sigma^*}(B)=0$ or $\prob_{\sigma^*}(B)=1$.
In the former case, choosing action 0 gives payoff 0, while choosing action 1 gives payoff 1,
and therefore the strategy $\sigma^*_i$ of each player $i$ chooses action 1 with probability at least $1 - \ep_i > 1/2$, which contradicts $\prob_{\sigma^*}(B)=0$. The argument is analogous for the latter case. Hence, $\sigma^*$ cannot be an $\vec\ep$-equilibrium.

Consider now the following repeated version $\Gamma$ of Voorneveld's game:
\begin{itemize}
\item   The set of players is $I = \dN$.
\item   The action set of each player $i \in I$ is $A_i = \{0,1\}$.
\item   The payoff of each player~$i\in I$
is equal to 1 if $r_i(a^t) = 1$ for all sufficiently large stages $t$, and it is equal to $0$ otherwise,
where $r_i$ is defined in \eqref{equ:1}.
\end{itemize}
Thus, player~$i$ wins in $\Gamma$ if she wins the one-shot Voorneveld game in all but finitely many stages.
The payoff function of this game is product-measurable and tail-measurable, and
hence the game satisfies the conditions of Theorem~\ref{theorem:1}.

Even though the one-shot Voorneveld game has no $\ep$-equilibrium for small $\ep>0$,
according to Theorem~\ref{theorem:1} the repeated game $\Gamma$ does have an $\ep$-equilibrium for every $\ep > 0$.
In fact, the following strategy profile is a pure 0-equilibrium in the game $\Gamma$: each player~$i$ plays action 0 in stages $t=0,1,\dots,i-1$ and action 1 in all stages $t\geq i$. Under this strategy profile, in each stage only finitely many players play action 1, and hence each player $i$ wins the one-shot Voorneveld game in all stages $t\geq i$,
i.e., $r_i(a^t) = 1$ for each $t\geq i$. Thus, under this strategy profile, each player's payoff is 1 in the game $\Gamma$.

This example illustrates the general construction.
Given $\ep > 0$, we will identify a play $p\in A^\dN$ such that $f_i(p)$ is at least player~$i$'s minmax value minus $\ep$, for each $i \in I$.
To find such a play $p$, the key step is to construct a strategy profile satisfying the following property:
each player $i$ plays arbitrarily in the first $i-1$ stages, and afterward she plays in a way that guarantees her a payoff of at least her minmax value minus $\ep$ with probability 1.
This way, in every stage, only finitely many players are ``active''.
Such a strategy profile will be constructed using techniques that were developed by Martin \cite{martin1998determinacy}.
From the existence of such a strategy profile it will follow that a play $p$ with the desired property does exist.

Based on the play $p$, we can define a $2\ep$-equilibrium: The players are supposed to follow $p$. If a player deviates from $p$, then she is punished,
i.e., her opponents switch to a strategy profile which ensures that player $i$'s payoff is at most her minmax value plus $\ep$.

To apply the results of Martin \cite{martin1998determinacy}, 
we cannot allow countably many players to randomize at the same stage.
We therefore do not use the classical minmax value of player $i$, where all opponents may randomize their actions, but rather we will define a new concept, the \emph{finitistic minmax value} of player $i$,
where the opponents of player~$i$ are restricted to strategy profiles with finite support. 

\section{The Proof}\label{sec3}
Section \ref{sec:minmax} introduces the new notion of finitistic minmax value, designed specifically for games with countably many players. We discuss the new notion and compare it to the ``classical'' minmax value in one-shot as well as in repeated games. Sections~\ref{sec:surr} and \ref{sec:suff} are devoted to two auxiliary results, both on the finitistic minmax value. The first of these results is largely based on the arguments of Martin \cite{martin1998determinacy} and Maitra and Sudderth \cite{maitra1998finitely} and 
relates the finitistic minmax value with
an auxiliary game of perfect information used to establish determinacy of Blackwell games. 
%Roughly speaking, this winning strategy in the auxiliary game could be thought of as an assignment of continuation values that ``proves'' player $i$'s finitistic minmax value in the game $\Gamma$ to be at least $\Val_i^F(\Gamma) - \delta$. 
The second result fine-tunes the first, with a goal of obtaining  continuation values that are good in each subgame. The final subsection, \ref{section:prooff}, brings all the arguments together and completes the proof of Theorem \ref{theorem:1}.

\subsection{Two Notions of Minmax}\label{sec:minmax}
In this subsection we discuss two notions of the minmax value that are important for the proof: the classical minmax value and the so-called finitistic minmax value.\medskip   

\noindent\textbf{The minmax value in one-shot games.} Consider a one-shot game $G=(I,(A_i)_{i\in I},(r_i)_{i\in I})$, where $I=\dN$ is a countable set of players, 
and for each $i \in I$, $A_i$ is a nonempty and finite action set of player $i$ and $r_i:A\to\dR$ is her payoff function. A strategy of each player $i$ is an element of $X_i$. 

Consider a player $i \in I$. Assume that her payoff function $r_i$ is bounded and $\mathcal{F}(A)$-measurable. This ensures that each strategy profile $x=(x_i)_{i\in I}$ induces an expected payoff $r_i(x)$ for player $i$. In the game $G$, the (classical) \emph{minmax value} of player $i$ is defined as
\begin{equation}\label{val-oneshot}
\val_i(r)\,=\,\inf_{x_{-i}\in X_{-i}}\sup_{x_i\in X_i} r_i(x_i,x_{-i}).
\end{equation}
Since $A_i$ is finite, we can replace $\sup_{x_i\in X_i}$ by $\max_{a_i\in A_i}$. 
The minmax value can be interpreted as the highest payoff that player $i$ can defend against any strategy profile of her opponents.

The classical minmax value is not always well defined when $r_i$ is not $\mathcal{F}(A)$-measurable, yet payoff functions arising in our proof are not guaranteed to be $\mathcal{F}(A)$-measurable. For this reason we introduce a related notion, the so-called finitistic minmax value. This notion imposes that all but finitely many of $i$'s opponents play a pure strategy. As a consequence, it is well defined whether or not the payoff function is $\mathcal{F}(A)$-measurable.

The relationship between the classical and the finitistic versions of the minmax value is explored in Section \ref{secn.extension}. In fact, we show (Theorem \ref{minmax-oneshot}) that the two notions coincide for any function $r_{i}$ that is bounded and $\mathcal{F}(A)$-measurable. Thus, for the opponents of player $i$ to minimize player $i$'s expected payoff, only finitely many of them need to randomize. As a consequence of this result, the finitistic minmax value can be seen as an extension of the classical minmax value to a larger class of payoff functions.  

We note that in \eqref{val-oneshot} we cannot replace the infimum by minimum. Indeed, suppose that $i=0$, player $0$ has a single action, while all other players have two actions: $a$ and $b$. Define $r_0$ to be $1/n$ if player $n>0$ is the only opponent who plays action $b$, and define $r_0$ to be 1 otherwise. Then player 0's minmax value and finitistic minmax value are equal to 0, but player 0's payoff is never 0.\medskip

\noindent\textbf{The minmax value in repeated games.} Let $\Gamma$ be a repeated game with countably many players. Assume that player $i$'s payoff function $f_i$ is bounded and product-measurable. Then, in the game $\Gamma$ the \emph{minmax value} of player $i$ is defined as
\begin{equation}\label{minmax-rep} 
\Val_i(\Gamma)=\inf_{\sigma_{-i} \in \Sigma_{-i}}\,\sup_{\sigma_i \in \Sigma_i}\E_{(\sigma_i,\sigma_{-i})}[f_i]. 
\end{equation}
Note our use of ``\val'' for the minmax value in one-shot games, and of ``\Val'' for the minmax value in repeated games: the parameter of the operator $\val_i$ is player~$i$'s payoff function, while the parameter of the operator $\Val_i$ is the whole game.

Still assuming $f_i$ to be bounded and product-measurable, in the game $\Gamma$ define the \emph{finitistic minmax value} of player $i$ as
\begin{equation}\label{surr-minmax-rep}
\Val_i^{F}(\Gamma) = \inf_{\sigma_{-i} \in \Sigma_{-i}^F}\,\sup_{\sigma_i \in \Sigma_i}\E_{(\sigma_i,\sigma_{-i})}[f_i]. 
\end{equation}

We do not know whether the finitistic minmax value equals the classical minmax value for every product-measurable function. In particular, Theorem \ref{minmax-oneshot} on the two versions of the minmax value for one-shot games does not easily yield a proof that the finitistic minmax value equals the classical minmax value in repeated games. Yet, it is possible that the two notations in fact coincide whenever the latter is well defined. We circumvent this, rather non-trivial question, and build the proof of our main result on the finitistic version of the minmax value for repeated games.

Given a history $h \in H$, the subgame of $\Gamma$ starting at $h$ is the game $\Gamma^h = (I,(A_i)_{i \in I}, (f_{i,h})_{i \in I})$, where $f_{i,h} = f_i \circ s_h$ and $s_h : A^{\dN} \to A^{\dN}$ is given by $p \mapsto (h,p)$. 
Player $i$'s finitistic minmax value in the subgame starting at $h$ can be computed as  

\[\Val_i^F(\Gamma^h) =\inf_{\sigma_{-i} \in \Sigma_{-i}^F}\, \sup_{\sigma_i \in \Sigma_i} \E_{(\sigma_i,\sigma_{-i})}[f_i | h]. \]

When a player's payoff function is tail-measurable, her payoff is independent of the initial segment of the play. As a result, her finitistic minmax value is independent of the history.

\begin{theorem}\label{const-surr}
Let $\Gamma$ be a repeated game with countably many players. Assume that some player $i$'s payoff function $f_i$ is bounded, product-measurable, and tail-measurable. Then, $\Val_i^F(\Gamma^h) = \Val_i^F(\Gamma)$ holds for every history $h \in H$.
\end{theorem}

We note that a similar statement holds for the classical minmax value defined in \eqref{minmax-rep}.

\subsection{Defending the Finitistic Minmax Value}\label{sec:surr}
We introduce the first of our two auxiliary results.

Let $\Gamma = (I,(A_i)_{i \in I}, (f_i)_{i \in I})$ be a repeated game with countably many players. Fix a player $i \in I$. For a function $d : H \to \dR$ and a history $h\in H$, we write $d^h : A \to \dR$ to denote the function $a \mapsto d(h,a)$. 

\begin{theorem}\label{martin-for-surr}
Let $\Gamma$ be a repeated game with countably many players. Consider a player $i \in I$ and assume that her payoff function $f_i$ takes values in $[0,1]$ and is product-measurable. 
Suppose moreover that $0 < \Val_i^F(\Gamma)$ and let $0 < w < \Val_i^F(\Gamma)$. There exists a function $d: H \to [0,1]$ satisfying the following conditions:
\begin{itemize}
\item[\rm(a)] At the empty history: $d(\emptyset) = w$.
\item[\rm(b)] $d(h) \leq \val_i^F(d^h)$ for every $h \in H$.
\item[\rm(c)] $\limsup_{t \to \infty} d(h^t(p)) \leq f_i(p)$ for every play $p\in A^\dN$.
\end{itemize}
\end{theorem}

Intuitively, the function $d$ assigns a fictitious continuation payoff to each subgame of $\Gamma$ in a way that  ``proves'' player $i$'s finististic minmax to be at least $w$. Condition (a) is the initialization. Condition (b) requires that, at any history $h$, if the opponents of player $i$ use a mixed strategy profile $x_{-i}$ with finite support, then player $i$ has a response $x_i$ such that the expectation of the function $d^h$ under the mixed action profile $(x_{i},x_{-i})$ is at least $d(h)$. In other words, player $i$ is able to defend the payoff of $d(h)$ in the one-shot game. Condition (c) states that the sequence of continuation values is at most the actual payoff. We give two examples.

\begin{example}\label{exl:voorneveld}\rm
Consider the example of Section~\ref{sec2}.\ref{subsecn.example}. Define $d$ recursively: let $w \in (0,1)$ and set $d(\emptyset) = w$. Let $h \in H$ be such that $d(h)$ has been defined. If $d(h) = 0$, let $d(h,a) = 0$ for each $a \in A$. If $d(h) > 0$, let $d(h,a) = r_i(a)$ for each $a \in A$. The function $d$ satisfies the conditions of the theorem since $\val_i^F(r_i) = 1$. $\Diamond$
%Let $p$ be the play on the equilibrium path: at stage $t \in \dN$, each player $i \in \{0, \dots, t\}$ plays 1, and each player $i \in \{t+1,t+2,\ldots\}$ plays 0. Define $d(h)$ to be $1$ or $0$, depending on whether $h$ is a prefix of the equilibrium play $p$.
\end{example}

\begin{example}\label{exl:0sum}\rm
Let $\Gamma$ be the following two-player zero-sum repeated game: The action set is $A_1=\{T,B\}$ for player 1 and $A_2=\{L,R\}$ for player 2 (and all other players are dummy). The payoff function $f = f_1$ assigns payoff 1 to each play $p\in(A_1\times A_2)^\dN$ such that $(T,L)$ or $(B,R)$ are played infinitely many times, and assigns payoff 0 to all other plays. One can think of $\Gamma$ as a repeated matching pennies game
\[
\begin{blockarray}{ccc}
 & L & R\\
\begin{block}{c(cc)}
  T & 1 & 0 \\
  B & 0 & 1 \\
  \end{block}
\end{blockarray}\vspace{-0.2cm}
\]
where player 1 wins if she receives payoff 1 infinitely often. The payoff function $f$ is tail-measurable. 
The value of $\Gamma$ (and of $\Gamma^h$ for each history $h$) is 1,
and one optimal strategy for player~1 is playing the mixed action $(\tfrac{1}{2},\tfrac{1}{2})$ at each stage. 

Let $\psi_{-},\psi_{+}:[0,1]\to[0,1]$ be the functions $\psi_{-}(x)=x-\frac{1}{2}\min\{x,1-x\}$ and $\psi_{+}(x)=x+\frac{1}{2}\min\{x,1-x\}$. Define $d$ recursively. Let $d(\emptyset) = w$. If $d(h)$ has been defined, let $d^h : A \to [0,1]$ be given as follows:
\begin{equation}\label{equ:exl0sum}
\begin{blockarray}{ccc}
 & L & R\\
\begin{block}{c(cc)}
  T & \psi_{+}(d(h)) & \psi_{-}(d(h)) \\[0.2cm]
  B & \psi_{-}(d(h)) & \psi_{+}(d(h)) \\
  \end{block}
\end{blockarray}\vspace{-0.2cm}
\end{equation}
Note that the value of $d^h$ is exactly $d(h)$. Consider a play $p$ where $(T,L)$ or $(B,R)$ are played only finitely many times. Then $d(h^{t}(p))$ is monotonically decreasing after some stage, and is approaching $0$. $\Diamond$
\end{example}

The proof of Theorem \ref{martin-for-surr} is a variation of the proof of Blackwell determinacy in Martin \cite{martin1998determinacy} and in Maitra and Sudderth \cite{maitra1998finitely}. The main tool of the proof is a perfect information game of two players, I and II, in which player I's goal is to ``prove'' that the minmax value in the game $\Gamma$ is at least $w$. In fact, the function $d$ is one of player I's winning strategies in that game. 
We provide a sketch of the argument, noting what modifications to Martin's proof are required, but omit the details.\medskip

\noindent\begin{proof}[Proof  of Theorem \ref{martin-for-surr}]
Given a game $\Gamma$, a player $i$ satisfying the conditions of the theorem, and a number $w \in (0,1]$, consider the following perfect information game, denoted $\calM_{i}(\Gamma,w)$:
\begin{itemize}
    \item Player~I chooses a one-shot payoff function $r_0 : A \to [0,1]$ such that $\val_i^F(r_0) \geq w$.
    \item Player~II chooses an action profile $a^0 \in A$ such that $r_0(a^0) > 0$.
    \item Player~I chooses a one-shot payoff function $r_1:A \to [0,1]$ such that $\val_i^F(r_1) \geq r_0(a^0)$.
    \item Player~II chooses an action profile $a^1 \in A$ such that $r_1(a^1)>0$.
    \item Player~I chooses a one-shot payoff function $r_2 : A \to [0,1]$ such that $\val_i^F(r_2) \geq r_1(a^1)$, and so on.
\end{itemize}
To distinguish between histories and plays in $\Gamma$ and in $\calM_i(\Gamma,w)$, we call the latter \emph{positions} and \emph{runs}.
The outcome of $\calM_i(\Gamma,w)$ is a run $(r_0,a^0,r_1,a^1,\ldots,)$. Player I wins $\calM_{i}(\Gamma,w)$ if 
\[\limsup_{t \to \infty}r_{t}(a^{t}) \leq f_i(a^0,a^1,\ldots).\]

Note that player~I has a legal move at any stage of the game, for instance the function $r : A \to \dR$ such that $r(a)=1$ for all $a \in A$. By induction one can show that player~II always has a legal move as well. Let $T$ be the set of all positions in the game $\calM_{i}(\Gamma,w)$. Then $T$ is a pruned tree on the set $R \cup A$, where $R$ denotes the set of functions $r : A\to\dR$. One can check that player I's winning set in $\calM_{i}(\Gamma,w)$ is a Borel subset of the set of runs,
where the set of runs is endowed with the product topology and $R \cup A$ is endowed with the discrete topology. By Martin \cite{martin1975borel}, $\calM_{i}(\Gamma,w)$ is determined: either player~I has a winning strategy in the game, or player~II does.

To prove the theorem, one establishes the following fact.
\begin{lemma}\label{lemma:PlayerI}
Player \I~has a winning strategy in $\calM_{i}(\Gamma,w)$ whenever $0 < w < \Val_i^F(\Gamma)$.
\end{lemma}

%\begin{lemma}\label{lemma:PlayerII}
%Player \II~has a winning strategy in $\calM_{i}(\Gamma,w)$ whenever $\Val_i^F(\Gamma) < w$.
%\end{lemma}

First we complete the proof of Theorem \ref{martin-for-surr} using the lemma. Take $0 < w < \Val_i^F(\Gamma)$ and let $\sigma_\I$ be a winning strategy of player I in $\calM_{i}(\Gamma,w)$. 
The strategy $\sigma_\I$ induces the function $d$ with the desired properties, as follows. Define $d(\emptyset) = w$. 
Let $T_\I \subseteq T$ denote the set of positions in the game $\calM_{i}(\Gamma,w)$ of even length (i.e., player I's positions) that are consistent with $\sigma_\I$. Let $H_\I \subseteq H$ denote the set of histories of $\Gamma$ consisting of the empty history $\emptyset$, and non-empty histories $(a^0,\ldots,a^{t-1})$ for which there is a sequence $r_0,\ldots,r_{t-1}$ of elements of $R$ such that $h^* = (r_0,a^0,\ldots,r_{t-1},a^{t-1})$ is an element of $T_\I$. Notice that $h^*$ is uniquely determined by $h \in H_\I$. If $h \in H \setminus H_\I$, define $d(h,a) = 0$ for each $a \in A$. If $h \in H_\I$, let $r_{t} = \sigma_\I(h^*)$ be player I's move at position $h^*$, and let $d(h,a) = r_{t}(a)$ for each $a \in A$.
The reader can verify that $d$ satisfies (a)--(c).

%That $d$ satisfies (a), (c), and (d) follows by the fact that $\sigma_\I$ is player I's winning strategy. To see that it satisfies (b), notice the function $d$ induces a winning strategy for player I in the subgame $\calM_{i}(\Gamma^h,d(h))$. Hence (b) follows by Lemma \ref{lemma:PlayerI}. This completes the proof of Theorem \ref{martin-for-surr}.\smallskip

The proof of the lemma is similar to the proof of the analogous result in Martin \cite{martin1998determinacy} and Maitra and Sudderth \cite{maitra1998finitely}. The only delicate point concerns the measurability property of strategies: for each stage $t$, player $i$, and action $a_i$, the mapping $h \mapsto \sigma_i(h)(a_i)$ should be measurable. Since in the lemma we restrict attention to strategy profiles with finite support, the number of histories that occur with positive probability at each stage is finite. Since the definition of the strategy in histories that occur with probability 0 is irrelevant, we can set the action chosen at these histories to be an arbitrary fixed action, ensuring that the strategies constructed in the proof of the lemma are measurable.
\end{proof}

\subsection{Defending the Finitistic Minmax Value in Each Subgame}\label{sec:suff}
We have interpreted the function $d$ of Theorem \ref{martin-for-surr} as a ``proof'' that the (finitistic minmax) value of player $i$ in $\Gamma$ is at least $w$. In this subsection we are after a function $d$ that could serve as a ``proof'' that player $i$'s finitistic minmax value in every subgame $\Gamma^h$ of $\Gamma$ is at least $\Val_i^F(\Gamma^h) - \delta$. We thus modify the conditions the function $d$ should satisfy. Condition (b) remains, and a new requirement is added: that the finitistic minmax value of the one-shot game $d^h$ be at least $\Val_i^F(\Gamma^h) - \delta$, for every history $h$. Moreover, we replace condition (c) with a weaker requirement. 

\begin{theorem}\label{theorem:martin}
Let $\Gamma$ be a repeated game with countably many players. Consider a player $i \in I$ and assume that her payoff function $f_i$ takes values in $[0,1]$ and is product-measurable. Then, for every $\delta > 0$ there exists a function $d : H \to [0,1]$ satisfying
the following conditions:
\begin{enumerate}
%\item[(M.1)] $d(h) \leq \Val_i^F(\Gamma^h)$ for every $h \in H$.
\item[(M.1)] $d(h) \leq \val^F_{i}(d^h)$ for every $h \in H$.
\item[(M.2)] $\Val_i^F(\Gamma^h) - \delta \leq \val^F_{i}(d^h)$ for every $h \in H$. 
\item[(M.3)] For every strategy profile with finite support $\sigma \in \Sigma^F$ that satisfies
\begin{equation}\label{equ:pr-m3} 
\val^F_{i}(d^h) \,\leq\, \E_{\sigma(h)}[d^h] \ \ \ \forall h \in H,
\end{equation}
we have
\[\val^F_{i}(d^h) \,\leq\, \E_\sigma[f_i | h]\ \ \ \forall h \in H. \]
\end{enumerate}
\end{theorem}

\begin{example}\rm
The function $d$ defined in Example \ref{exl:voorneveld} violates (M.2), since there are histories $h$ such that the one-shot payoff function $d^h$ is identically zero, whereas $\Val_i^F(\Gamma^h) = 1$ for each history $h$. 

Consider instead the function $d$ defined by letting $d(\emptyset) = 1$ and $d(h,a) = r_i(a)$ for each $h \in H$ and $a \in A$. Thus $d^h = r_i$, and so $\val_i^F(d^h) = 1$ for each $h$. Conditions (M.1), (M.2), and (M.3) are all satisfied. To verify (M.3), consider a strategy profile $\sigma \in \Sigma^F$ satisfying \eqref{equ:pr-m3}. Then for each $h \in H$ the mixed action profile $\sigma(h)$ is supported by the set of action profiles $a \in A$ such that $r_i(a) = 1$. This implies that $\prob_\sigma(\cdot \mid h)$ is supported by the set of plays $p \in A^\dN$ with $f_i(p) = 1$. 

Note however that the function $d$ does not satisfy condition (c) of Theorem \ref{martin-for-surr}: consider for instance the play $p$ where at each stage all the players in $-i$ play $0$, whereas player $i$ plays $0$ at even stages and $1$ at odd stages. Then $i$'s reward is $0$ at even stages and $1$ at odd stages, implying that $f_i(p) = 0$ whereas $\limsup_{t \to \infty}d(h^t(p)) = 1$. In fact, one can show that there is no function $d$ that simultaneously satisfies conditions~(M.2) and~(M.3) of Theorem \ref{theorem:martin} and condition (c) of Theorem \ref{martin-for-surr}.\medskip
%
%\AP{Here's the argument, to be deleted: Let $S$ be the set of plays $p$ such that $\limsup_{t \to \infty}d(h^t(p)) \geq 1/2$. This is a $G_\delta$ set contained in $\{p : f(p) = 1\}$, a countable set. Hence $S$ cannot be dense. Hence there is a history $g$ such that $S$ is disjoint from $[g]$. Now let $R$ be the set of histories $h$ extending $g$ such that $d(h) \geq 1/2$. It cannot be the case that for each $h \in R$ there is an $a \in A$ such that $(h,a) \in R$, for if that were the case, there would exist a play $p$ with its every prefix, starting with $g$, an element of $R$. Hence there is a history $h$ such that $(h,a) \notin R$ for each $a \in A$. But then  $d^h$ has value less than $1/2$.}
\end{example}

\begin{example}\rm
The function $d$ constructed in Example \ref{exl:0sum} fails to satisfy condition (M.2): we have seen that $\val_i^F(d^h)$ can be arbitrarily close to zero. 

Fix $\delta \in (0,\frac{2}{3})$, and
consider instead a function $d$ defined recursively as follows. Let $d(\emptyset) = 1 - \delta$. Take $h \in H$ such that $d(h)$ has been defined. If $1-\delta \leq d(h)$, define $d^h : A \to [0,1]$ to be given by \eqref{equ:exl0sum}; if $d(h) < 1-\delta$, define $d^h$ to be the function
\[
\begin{blockarray}{ccc}
 & L & R\\
\begin{block}{c(cc)}
  T & 1 - \tfrac{1}{2}\delta & 1 - \tfrac{3}{2}\delta \\[0.2cm]
  B & 1 - \tfrac{3}{2}\delta & 1 - \tfrac{1}{2}\delta \\
  \end{block}
\end{blockarray}.\vspace{-0.2cm}
\]
Thus, if $1-\delta \leq d(h)$, then $d(h) = \val_i^F(d^h)$; and if $d(h) < 1 - \delta$, then $1 - \delta = \val_i^F(d^h)$. This shows that both (M.1) and (M.2) are satisfied. To verify (M.3), let $\sigma$ satisfy \eqref{equ:pr-m3}. Then, for each $h \in H$, the measure $\sigma(h)$ places probability of at least $\tfrac{1}{2}$ on the set $\{(T,L),(B,R)\}$. 
This implies that $a^t \in \{(T,L),(B,R)\}$ for infinitely many $t \in \dN$ almost surely with respect to the measure $\prob_\sigma(\cdot\mid h)$. Thus $\E_\sigma[f_i | h] = 1$. 
\end{example}

\noindent\begin{proof}[Proof of Theorem~\ref{theorem:martin}]
Fix $\delta > 0$. \medskip

\noindent\textsc{The function $\widehat D_{h}$.} Consider an arbitrary history $h\in H$. Let $H_h\subseteq H$ be the set of all histories that extend $h$ (including $h$ itself). We define a function $\widehat D_{h} : H_h \to [0,1]$ as follows.

Suppose first that $\Val_i^F(\Gamma^h) \geq \delta/2$. Apply Theorem \ref{martin-for-surr} to the subgame $\Gamma^h$ and $w = \Val_i^F(\Gamma^h) - \delta/2$. This yields a function $\widehat D_{h} : H_h \to [0,1]$ with the following properties:
\begin{itemize}
\item At the history $h$:
\begin{equation}\label{equ:initath}
\widehat D_{h}(h) = \Val_i^F(\Gamma^h) - \tfrac{\delta}{2}.
\end{equation}
\item For every $g \in H_h$: 
\begin{equation}
\label{equ:m1} 
\widehat D_{h}(g) \leq \val_i^F(\widehat D_{h}^{g}).
\end{equation}
%\item For every $g \in H_h$: 
%\begin{equation}
%\label{equ:m2} 
%\widehat D_{h}(g) \leq \Val_i^F(\Gamma^{g}).
%\end{equation}
%\item If $\widehat D_{h}(h') = 0$ then $\widehat D_{h}(h',a) = 0$ for every $h' \in H_h$ and every $a \in A$.
\item For every play $p$ that extends $h$ we have 
\begin{equation}\label{equ:pl1wins}
\limsup_{t \to \infty} \widehat D_h(h^t(p)) \leq f_i(p).
\end{equation}
\end{itemize}
If $\Val_i^F(\Gamma^h)< \delta/2$, then we let $\widehat D_{h}$ be the constant zero function. This definition satisfies \eqref{equ:m1}
%, \eqref{equ:m2}, 
and \eqref{equ:pl1wins}, whereas \eqref{equ:initath} has to be replaced by  
\begin{equation}\label{equ:lessthandelta}
\Val_i^F(\Gamma^h) - \tfrac{\delta}{2} \leq \widehat D_{h}(h) = 0.
\end{equation}

\noindent\textsc{The function $d$.} We recursively define a function $d : H \to [0,1]$ and, simultaneously, an auxiliary function $\alpha:H \to H$, which assigns to each history $h$ a prefix $\alpha(h)$ of $h$. 
\begin{itemize}
\item  At the empty history: Set $\alpha(\emptyset)= \emptyset$ and  $d(\emptyset) = \widehat D_{\alpha(\emptyset)}(\emptyset)=\widehat D_{\emptyset}(\emptyset) = \Val_i^F(\Gamma)-\delta/2$.
\item  Now consider a history $h \neq \emptyset$, and suppose that $\alpha(g)$ and $d(g)$ are already defined for each prefix $g$ of $h$, with $g\neq h$. Let $h^-$ denote the prefix of $h$ up to the previous stage.
\begin{itemize}
\item Set $d(h) = \widehat D_{\alpha(h^-)}(h)$.
\item If $\Val_i^F(\Gamma^h) - \delta \leq d(h)$, set $\alpha(h)=\alpha(h^-)$. Otherwise, set $\alpha(h)= h$; in this case we say that $d$ \emph{re-initiates} at $h$.
\end{itemize}
\end{itemize}

The intuition behind the definition of the function $d$ is the following. We start at $h_0 = \emptyset$ with the function $\widehat D_{h_0}$ and follow it until we encounter a history $h_1$, say in period $t_1$, where $\widehat D_{h_0}(h_1) < \Val_i^F(\Gamma^{h_1}) - \delta$. From this point on, the function $\widehat D_{h_0}$ is no longer useful, and re-initiation occurs, which takes effect from the next stage, i.e., stage $t_1 + 1$. As of stage $t_1 + 1$, $d$ follows the function $\widehat D_{h_1}$, until encountering a history $h_2$, say in period $t_2$, such that $\widehat D_{h_1}(h_2) < \Val_i^F(\Gamma^{h_2}) - \delta$. As of stage $t_2 + 1$, $d$ follows $\widehat D_{h_2}$, and so on.

We show that $d$ satisfies the properties (M.1)--(M.3).\medskip

\noindent\textsc{Verifying (M.1) and (M.2).} For each history $h\in H$, as $d(h,a)=\widehat D_{\alpha(h)}(h,a)$ for all action profiles $a\in A$, we have equivalently 
\begin{equation}\label{eqn.defined}
d^h = \widehat D_{\alpha(h)}^h.
\end{equation}
For the empty history $\emptyset$, by \eqref{equ:initath}, \eqref{equ:lessthandelta}, \eqref{equ:m1}, and \eqref{eqn.defined},
\begin{align*}
\Val_i^F(\Gamma) - \delta \,&<\, \Val_i^F(\Gamma) - \tfrac{\delta}{2}\, \leq\,\widehat D_{\emptyset}(\emptyset)\,=\,d(\emptyset)\\ & 
\leq\, \val_i^F(\widehat D_{\emptyset}^\emptyset)\,=\,\val_i^F(d^\emptyset),
\end{align*}
which proves (M.1) and (M.2) for $\emptyset$.

Consider now a history $h\neq\emptyset$ at which $d$ does not re-initiate. 
Since $\alpha(h)=\alpha(h^-)$, by \eqref{equ:m1}, and by \eqref{eqn.defined},
\begin{align*}
\Val_i^F(\Gamma^h) - \delta\, &\leq \,d(h) \,=\, \widehat D_{\alpha(h^-)}(h) \,=\, \widehat D_{\alpha(h)}(h)\\
&\leq\, \val_i^F\big(\widehat D_{\alpha(h)}^h\big) \,=\, \val_i^F(d^h),
\end{align*}
which proves (M.1) and (M.2) for $h$.

Finally, consider a history $h\in H$ at which $d$ re-initiates. By  \eqref{equ:initath}, \eqref{equ:lessthandelta}, \eqref{equ:m1},  \eqref{eqn.defined}, and since $\alpha(h) = h$,
\begin{align}
d(h) \,&<\, \Val_i^F(\Gamma^h) - \delta
\,<\, \Val_i^F(\Gamma^h) - \tfrac{\delta}{2}\nonumber\\
&\leq\, \widehat D_{h}(h) \,\leq\, \val^F_i(\widehat D_{h}^h)\,=\, \val^F_i(d^h)
,\label{equ:jump}
\end{align}
which proves (M.1) and (M.2) for $h$.\medskip

\noindent\textsc{Verifying (M.3).} Assume that a strategy profile with finite support $\sigma \in \Sigma^F$ satisfies \eqref{equ:pr-m3}. Take a history $g \in H$. We show that $\val^F_i(d^g) \leq \E_\sigma[f_i | g]$. 

Notice that \eqref{equ:pr-m3} and~(M.1) imply that $d(h) \leq \E_{\sigma(h)}[d^h]$ for each history $h\in H$, and thus the process $(d(h))_h$ is a submartingale under $\sigma$ (with respect to the usual filtration on $A^\dN$). Notice also that if $d$ re-initiates at a history $h\in H$, then by \eqref{equ:jump} and \eqref{equ:pr-m3} we have $d(h)+\tfrac{\delta}{2} \leq \val^F_i(d^h) \leq \E_{\sigma(h)}[d^h]$. This implies that under $\sigma$ the
expected number of re-initiations is bounded in any subgame. In particular
\[\prob_\sigma(\hbox{the number of re-initiations is finite }\mid g) = 1,\]
for every history $g$.
Fix $\rho > 0$, and let $T \in \dN$ be a sufficiently large stage such that
\[\prob_\sigma(\hbox{no re-initiation after stage } T \mid g)\, \geq\, 1-\rho. \]
This implies 
\[\prob_\sigma(\alpha(h^T(p)) = \alpha(h^{t}(p))\text{ for each }t\geq T\mid g)\, \geq\, 1-\rho.\] 
Hence, 
\begin{align*}
\val^F_i(d^g) \,&\leq\, \E_{\sigma(g)}[d^g]\\
\,&\leq\, \E_\sigma\Big[\limsup_{t \to \infty} d(h^t)\mid g \Big]\\
&\leq\, \E_\sigma\Big[\limsup_{t \to \infty} \widehat D_{\alpha(h^T)}(h^t) \mid g \Big] + \rho\\
&\leq\, \E_\sigma[f_i | g] + \rho,
\end{align*}
where the last inequality follows from \eqref{equ:pl1wins}, and thus~(M.3) holds as well.
\end{proof}

\begin{comment}
In the proof of Theorem~\ref{theorem:martin} the function $d$ is updated whenever its value drops too much, that is, whenever its continuation cannot serve to show that the value of the current subgame is close to $\Val_i^F(\Gamma)$.
This idea is a reminiscent of, e.g., \cite{rosenberg2000maxmin}, where to construct an optimal strategy one follows a discounted optimal strategy until a subgame is reached where the discounted value is far below the undiscounted value, and then switches to a discounted optimal strategy with a higher discount factor that is closer to the undiscounted value. 
\end{comment}

\subsection{Proof of Theorem~\ref{theorem:1}}
\label{section:prooff}

In this section we complete the proof of Theorem~\ref{theorem:1}.
By applying a proper affine transformation to the payoff function of each player~$i$ and adjusting $\ep_i$ accordingly, we can assume w.l.o.g.~that the range of $f_i$ is $[0,1]$.
The gist of the argument is as follows.
For each player $i \in I$, let $d_i : H \to [0,1]$ be a function satisfying the conditions of Theorem \ref{theorem:martin} with $\delta = \frac{\ep_i}{2}$, and fix arbitrarily some default action $\widehat{a}_i \in A_i$. 
We define, for each history $h \in H$, a one-shot game $G(h)$. This is a game with finitely many players: if $h$ is a stage $t$ history, then in $G(h)$ only the players $0, \ldots, t$ are active; the other players are assumed to be playing their respective default actions. The payoff to an active player $i$ is given by the function $d_i^h$. Nash equilibria of these one-shot games satisfy  \eqref{equ:pr-m3} for every player $i \in \{0,\dots,t\}$. This allows us to conclude that the corresponding strategy profile in the repeated game $\Gamma$ gives each player $i$, almost surely, a payoff of at least $\Val_i^F(\Gamma) - \tfrac{\ep_i}{2}$. This in turn implies that there exists a play $p$ with a payoff of at least $\Val_i^F(\Gamma) - \tfrac{\ep_i}{2}$ for each player $i \in I$. Such a play yields an $\vec{\ep}$-equilibrium in $\Gamma$.

We turn to the formal proof.
Fix $\vec{\ep} > 0$. For each player $i \in I$ fix an arbitrary action $\widehat a_i \in A_i$. Apply Theorem~\ref{theorem:martin} with $\delta = \tfrac{\ep_i}{2}$, and let $d_i : H \to [0,1]$ be a function satisfying the conditions (M.1)--(M.3).\medskip

\noindent\textbf{The one-shot game $G(h)$.} Fix a history $h$, and denote $t=\len(h)$. Let $G(h)$ be the one-shot game where:
\begin{itemize}
\item The set of players is $0,1,2,\dots,t$.
\item The action set of each player $i \in \{0,1,2,\dots,t\}$ is $A_i$.
\item The payoff function $r^h_i : \prod_{j=0}^t A_j \to [0,1]$ of each player $i \in \{0,1,2,\dots,t\}$ is given by
\[r^h_i(a_0, \ldots, a_t)\, =\, d_i^h(a_0, \ldots, a_t, \widehat{a}_{t+1}, \widehat{a}_{t+2}, \ldots).\]
\end{itemize}
%The interpretation of the game $G(h)$ is the following: the players $0,1,2,\dots,t$ are active and free to choose their actions; the other players are assumed to be playing their respective default actions. The function $d_i^h$ serves as the payoff function of an active player $i$.

By Nash \cite{nash1950equilibrium}, the one-shot game $G(h)$ has an equilibrium
$x(h) = (x_i(h))_{i=0}^t \in \prod_{i=0}^t X_i$. Since each player $i=0,1,2,\ldots,t$ is active in $G(h)$ and only finitely many of her opponents randomize, we have
\begin{equation}\label{ineq-Gh}
\E_{x(h)} [r_i^h]\, \geq\, \val_i^F(d^h)\quad\text{for each}\quad i \in \{0,\ldots,t\}.
\end{equation}

\noindent\textbf{The strategy profile $\sigma$ and the play $p$.} For $t \in \dN$ define the set $H_t^F \subseteq A^t$ as follows. Let $H_0^F$ consist of the empty history $\emptyset$ alone; and, for $t \in \dN$ let $H_{t+1}^F$ be the set of histories $(a^0,\ldots,a^{t}) \in A^{t+1}$ such that $a_{i}^{k} = \widehat{a}_{i}$ whenever $k \leq t$ and $k < i$. This is the set of histories that results if the inactive players stick to their default actions. Note that $H_t^F$ is finite for each $t \in \dN$. 
 
For each player $i \in I$, let $\sigma_i\in \Sigma_i$ be the following strategy: for each history $h\in A^t$ at stage $t$ define
\[\sigma_i(h) = \left\{
\begin{array}{cll}
x_i(h), & & \text{if }h \in H_t^F,\\
\widehat{a}_i, & \ \ \ \ \ & \text{otherwise}.
\end{array}
\right.\]
Note that the function thus defined \textit{is} a strategy, i.e., 
the map $h \mapsto \sigma_i(h)(a_i)$ is $\mathcal{F}(A^t)$-measurable for every $a_i \in A_i$. 
Indeed, this map coincides with the function $h \mapsto x_i(h)(a_i)$ on the finite set of histories $H_t^F$, and it is constant on $A^t \setminus H_t^F$. 
%This clearly guarantees the required measurability. 

The resulting strategy profile $\sigma=(\sigma_i)_{i\in I}$ is an element of $\Sigma^F$.
Consider any player $i\in I$. For every history $h \in H$ with $\len(h)\geq i$, we have, by \eqref{ineq-Gh},
\[
\E_{\sigma(h)}[d_i^h] \,=\, \E_{x(h)} [r_i^h] \,\geq\, \val_i^F(d^h).
\]
By Theorems~\ref{theorem:martin} and~\ref{const-surr} we deduce that for every history $h \in H$ with $\len(h)\geq i$,
\[\E_\sigma[f_i | h] \,\geq\, \val_i^F(d^h) \,\geq\,\Val^F_i(\Gamma^h) - \tfrac{\ep_i}{2} \,=\, \Val_i^F(\Gamma) - \tfrac{\ep_i}{2}.\]
Since this holds for all histories beyond stage $i$, L\'evy's 0-1 law implies that
\[ \prob_{\sigma}\bigl(f_i \geq \Val_i^F(\Gamma) - \tfrac{\ep_i}{2}\bigr) \,=\, 1. \]
Since the set $I$ of players is countable, it follows that there exists a play $p\in A^\dN$ such that
\[ f_i(p) \,\geq\, \Val^F_i(\Gamma)-\tfrac{\ep_i}{2}, \quad\forall i\in I. \]
\textbf{The $\vec{\ep}$-equilibrium in $\Gamma$.} Based on the play $p$, the following grim-trigger strategy profile is an $\vec{\ep}$-equilibrium:
\begin{itemize}
\item The players follow the play $p$.
\item If some player $i$ deviates from $p$ at stage $t$,
then from stage $t+1$ and on all other players punish player~$i$ at her finitistic minmax value.
That is, they bring down player~$i$'s payoff to $\Val_i^F(\Gamma) + \tfrac{\ep_i}{2}$.
\end{itemize}

\section{Extensions and related results}\label{secn.extension}
\subsection{Extension to the Product-Discrete Sigma-Algebra}
Our main result, Theorem \ref{theorem:1}, assumes the payoff functions to be product-measurable. The result can in fact be extended to payoff functions that are measurable with respect to a larger sigma-algebra, obtained by endowing the set $A$ of action profiles with its discrete sigma-algebra. The extension, however, requires a more general concept of equilibrium. 

\begin{definition} 
The product-discrete sigma-algebra $\calF^{d}(A^\dN)$ over $A^\dN$ is the product sigma-algebra, where each copy of $A$ is endowed with the discrete sigma-algebra. 
%A function that is measurable w.r.t.~$\calF^{d}(A^\dN)$ is called \emph{product-discrete measurable}. 
\end{definition}

The product-discrete sigma-algebra is not an esoteric construct. In fact, Martin's celebrated results on Borel determinacy (\cite{martin1975borel}, \cite{martin1998determinacy}) apply to sets in $\calF^{d}(A^\dN)$.     

The product-discrete sigma-algebra $\calF^{d}(A^\dN)$ coincides with the product sigma-algebra $\calF(A^\dN)$ if and only if $A_i$ is finite for all but finitely many players $i$.
%Thus, as soon as there are infinitely many non-dummy players in the game, $\calF^{d}(A^\dN)$ is strictly larger than $\calF(A^\dN)$. 
Consequently, the class of $\calF^{d}(A^\dN)$-measurable functions $f_i:A^\dN\to\dR$ is generally richer than the class of product-measurable functions. For example, any payoff function that only depends on the action profiles played at the first $t$ stages of the game is $\calF^{d}(A^\dN)$-measurable, but not necessarily product-measurable.

The product-discrete sigma-algebra $\calF^{d}(A^\dN)$ coincides with the Borel sigma-algebra on $A^\dN$ if $A$ is given its discrete topology and $A^\dN$ the corresponding product topology.

\begin{definition}
Consider a repeated game with infinitely many players. Assume that player $i$'s payoff function $f_i$ is bounded. For a strategy profile $\sigma$, we say that player $i$ has a well defined expected payoff under $\sigma$ if
\[\sup_{g\in M,\ g\leq f_i} \E_{\sigma}[g] \,=\,\inf_{g\in M,\ g\geq f_i} \E_{\sigma}[g],\] where $M$ denotes the set of all bounded and product-measurable functions from $A^\dN$ to $\dR$. In that case, this common quantity is denoted by $\E_{\sigma}[f_i]$.
\end{definition}

\begin{lemma}\label{welldef-exp} 
If $f_i$ is bounded and $\calF^{d}(A^\dN)$-measurable, then player $i$ has a well defined expected payoff under any strategy profile with finite support. 
\end{lemma}

\noindent\begin{proof}
Fix $\sigma\in\Sigma^F$,
and let $H^F$ be the set of all  histories that occur with positive probability under $\sigma$.
%. Define recursively the sequence of finite sets $H_{t}^F \subseteq A^{t}$. Let $H_{0}^{F}$ be the set consisting of the empty history $\emptyset$ alone. If $H_{t}^{F}$ has already been defined, let $H_{t+1}^{F}$ consist of histories $(h,a) \in H_{t}^{F} \times A$ such that $\sigma_i(h)(a_i) > 0$ for each $i \in I$. Let $H^{F} = \cup_{t \in \dN} H_{t}^{F}$. 
Let $[H^F]$ be the set of plays whose prefixes are in $H^F$. 
Clearly, $\prob_\sigma([H^F]) = 1$.

Assume w.l.o.g.~that $f_i(p) \in [0,1]$ for each $p \in A^{\dN}$. 
Let $g_* : A^\dN \to \dR$ be the function that agrees with $f_i$ on $[H^F]$ and is 0 on the complement of $[H^F]$.
Similarly, let $g^*$ be equal to $f_i$ on $[H^F]$ and to 1 on the complement of $[H^F]$. Then both $g_*$ and $g^*$ are in $M$, and $g_* \leq f_i \leq g^*$. Since $\E_\sigma[g_*] = \E_\sigma[g^*]$, we conclude that $f_i$ has a well defined expected payoff. 
\end{proof}

A strategy profile $\sigma^*$ is an $\vec\ep$-equilibrium if 
(a) it induces a well defined payoff,
(b) all individual deviations of the players induce well defined payoffs,
and (c) no player can profit more than $\ep_i$ by deviating.
Since in the proof of Theorem \ref{theorem:1} we only use strategy profiles with finite support, by Lemma \ref{welldef-exp} we obtain the following extension of Theorem \ref{theorem:1}.

\begin{theorem}
Consider a repeated game with countably many players. Assume that each player's payoff function is bounded, $\calF^d(A^\dN)$-measurable, and tail-measurable. Then for each $\vec \ep=(\ep_i)_{i\in I}>0$ the game admits a strategy profile $\sigma^*$ with the following properties for every player $i \in I$:
\begin{itemize}
\item For every strategy $\sigma_i \in \Sigma_i$, player $i$ has a well defined expected payoff under $(\sigma_i,\sigma^*_{-i})$. In particular, player $i$ has a well defined expected payoff under $\sigma^*$.
\item For every strategy $\sigma_i \in \Sigma_i$,
we have $\E_{\sigma^*}[f_i] \,\geq\, \E_{(\sigma_i,\sigma^*_{-i})}[f_i] - \ep_i$.
\end{itemize}
\end{theorem}

\subsection{Two notions of minmax in one-shot games}
We show that in a one-shot game, if player $i$'s payoff function is product-measurable, then player $i$'s finitistic minmax value equals her classical minmax value. This implies that the former is an extension of the latter to a larger class of payoff functions. 

\begin{theorem}\label{minmax-oneshot} Let
$G$ be a one-shot game as in Section~\ref{sec3}.\ref{sec:minmax}. Assume that player $i$'s payoff function $r_i$ is bounded and $\mathcal{F}(A)$-measurable. Then
\[\val_i(r_i)\,=\,\val_i^F(r_i).\]
\end{theorem}

Let us assume for concreteness that $i = 0$. Fix a bounded product-measurable function $r = r_0 : A \to \dR$, an $\varepsilon > 0$, and a strategy profile $x_{-0} \in X_{-0}$ of $0$'s opponents. Consider an action profile $a_{-0} = (a_{1},a_{2},\ldots) \in A_{1} \times A_{2} \times \cdots = A_{-0}$ of 0's opponents and an $n \in \dN$. Consider also the profile $(x_{1}, \ldots, x_{n}, a_{n+1}, a_{n+2}, \ldots) \in X_{-0}^F$ that requires the first $n$ opponents of player $0$ to play as in the profile $x_{-0}$, and the others to play instead their pure action in $a_{-0}$. We say that $a_{-0}$ is an $n$-\textit{finitistic approximation} for $(r, \varepsilon, x_{-0})$ if 
\[\left|\E_{a_{0},x_{1}, \ldots, x_{n}, a_{n+1}, a_{n+2}, \ldots}[r] - \E_{a_0, x_{-0}}[r]\right| \leq \ep,\]
for each $a_0 \in A_0$. We say that $a_{-0}$ is a \textit{finitistic approximation} for $(r, \varepsilon, x_{-0})$ if there is an $n \in \dN$ such that $a_{-0}$ is an $n$-finitistic approximation for $(r, \varepsilon, x_{-0})$.

To establish the theorem, it suffices to show that there is a finitistic approximation for $(r, \varepsilon, x_{-0})$. As a matter of fact, the lemma below says that there are plenty. 

\begin{lemma}\label{lemma:minmax-oneshot}
Let $r :  A \to \dR$ be a bounded Borel measurable function, $\varepsilon > 0$, and $x_{-0} \in X_{-0}$. The set of finitistic approximations for $(r, \varepsilon, x_{-0})$ has $\prob_{x_{-0}}$-measure $1$. 
\end{lemma}
\begin{proof}
For $n \in \dN$, let $E_{n}$ be the set of $n$-finitistic approximations for $(r, \varepsilon, x_{-0})$, and $E = \bigcup_{n \in \dN} E_{n}$ the set of finitistic approximations for $(r, \varepsilon, x_{-0})$. 

Let $\mathcal{E}_{n}$ be the sigma-algebra on $A_{-0}$ generated by the random variables $\{a_{n+1},a_{n+2},\ldots\}$. Then $\mathcal{E} = \bigcap_{n \in \mathbb{N}}\mathcal{E}_{n}$ is the tail sigma-algebra on $A_{-0}$. Given $a_0 \in A_0$, define the process $\xi_0^{a_0}, \xi_1^{a_0}, \ldots$ on $A_{-0}$ by $\xi_{0}^{a_0}(a_{-0}) = r(a_{0}, a_{-0})$, and for $n \geq 1$ by \[\xi_{n}^{a_0}(a_{-0}) = \E_{a_{0},x_{1}, \ldots, x_{n}, a_{n+1}, a_{n+2}, \ldots}[r].\] 
%\[\xi_{n}(a) = \E_{\overline{a}_{0}, x_{1}, \ldots, x_{n-1}, a_{n}, a_{n+1}, ,\ldots}[r].\]
%\[\xi_{n}^{a_0}(a_{-0}) = \int_{A_1 \times \cdots \times A_{n}} r(a_{0},\overline{a}_{1},\dots,\overline{a}_{n}, a_{n+1}, a_{n+2},\ldots) dx_1(\overline{a}_{1}) \cdots dx_n(\overline{a}_{n}).\]

For $n \in \dN$, the random variable $\xi_{n}^{a_0}$ is $\mathcal{E}_{n}$-measurable. Moreover, the process $\{\xi_{n}\}_{n \in \mathbb{N}}$ is a reverse martingale with respect to the sequence  $\{\mathcal{E}_{n}\}_{n \in \mathbb{N}}$ of filtrations and the measure $\prob_{x_{-0}}$. That is to say, $\mathbb{E}_{x_{-0}}[\xi_{n}^{a_0} \mid \mathcal{E}_{n+1}] = \xi_{n+1}^{a_0}$. By Durrett \cite{Durrett11probability} (Theorems 5.6.1 and 5.6.2), the sequence $\xi_{n}^{a_0}$ converges to an $\mathcal{E}$-measurable limit $\xi_{\infty}^{a_0}$, almost surely with respect to $\prob_{x_{-0}}$, and $\E_{x_{-0}}[\xi_{\infty}^{a_0}] = \E_{x_{-0}}[\xi_0^{a_0}] = \E_{a_{0}, x_{-0}}[r]$. Since $\prob_{x_{-0}}$ is a product measure, Kolmogorov's 0-1 law implies that $\xi_{\infty}^{a_{0}}$ is constant $\prob_{x_{-0}}$-almost surely. Hence $\prob_{x_{-0}}$-almost surely, $\xi_{n}^{a_0}$ converges to $\E_{a_0,x_{-0}}[r]$. Since the set $A_{0}$ is finite, $E$ has $\prob_{x_{-0}}$-measure 1. 
\end{proof}

One can prove Theorem~\ref{minmax-oneshot} without using reverse martingales, relying instead on Levy's 0-1 law. 

\section{Discussion}\label{sec4}

\textbf{The set of equilibrium payoffs.} Let $\cal W$ denote the set of equilibrium payoffs, that is, accumulation points of $\ep$-equilibrium payoffs
as $\ep$ goes to 0. The grim-trigger strategy designed in the proof of Theorem~\ref{theorem:1} shows that
if $p$ is a play such that $f_i(p) \geq \Val_i^F(\Gamma)$ for each player $i\in I$,
then $f(p)=(f_i(p))_{i\in I}$ belongs to $\cal W$.

Using jointly controlled lotteries (see, e.g., Aumann and Maschler \cite{aumann1995repeated}),
we can deduce that under our assumptions, $\cal W$ includes the convex hull of the set of feasible and finitistic individually rational payoffs:
\[
R\,=\,\text{conv}\left\{f(p) \colon \,p\in A^\dN,\,f_i(p) \geq \Val_i^F(\Gamma)\ \,\forall i \in I\right\},
\]
where $\text{conv}$ refers to all convex combinations of countably many points of the set.
%\begin{eqnarray*}
%&&\left\{ g = \sum_{k=1}^K \alpha_k f(p_k) \in \dR^\dN  \colon \right. \\
%&&\left.\ \ \ K \in \dN,
%\sum_{k=1}^K \alpha_k = 1, \alpha_k \geq 0, f_i(p_k) \geq v_i^F(\Gamma) \ \ \ k=1,2,\ldots,K, i \in I\right\},
%\end{eqnarray*}
However, $\cal W$ may contain points outside the set $R$.
It would be interesting to characterize the set $\cal W$.\medskip

\noindent\textbf{Construction of equilibrium based on continuation payoffs.} The construction of an $\ep$-equilibrium in repeated games involved
the determination of fictitious continuation payoffs,
which were represented by the function $d$, based on the results of Martin \cite{martin1998determinacy}.
The quantity $d_i^h$ approximates player~$i$'s finitistic minmax value at the history $h$,
and as soon as at each history $h$ the active players play
a Nash equilibrium in the one-shot game whose payoff function is given by $d^h = (d^h_i)_{i \in I}$,
all of them are guaranteed to obtain at least their finitistic minmax value.

The relation between our result and Martin's \cite{martin1998determinacy} technique for zero-sum games,
is similar to the relation between Solan and Vieille's \cite{solan2002correlated} proof
for the existence of an extensive-form correlated equilibrium in stochastic games with finitely many players
and Mertens and Neyman's \cite{mertens1981stochastic,mertens1982stochastic}) study of zero-sum stochastic games.
Mertens and Neyman proved that to play optimally in a zero-sum stochastic game,
the player can play, at each history,
a discounted-optimal strategy in a one-shot game,
where the continuation payoff is given by the discounted value of the continuation game,
and the discount factor changes dynamically along the play as a function of past performance.
To prove the existence of an extensive-form correlated equilibrium in stochastic games with finitely many players,
Solan and Vieille \cite{solan2002correlated} defined a collection of one-shot discounted games,
where the discount factor of each player,
as well as her payoff functions,
are the ones given by Mertens and Neyman \cite{mertens1981stochastic,mertens1982stochastic}) for that player.
The tail-measurability of the payoff function in our model allows us
to deduce the existence of an $\ep$-equilibrium,
rather than of an extensive-form correlated equilibrium.
When the payoff function is product-measurable, and not necessarily tail-measurable,
our technique proves the existence of an extensive-form correlated equilibrium in repeated games
with finitely many players,
which provides an alternative and simpler proof than the one given by Mashiah-Yakovi \cite{mashiah2015correlated}.
\medskip

\noindent\textbf{Tail-measurability.} In our proof, the tail-measurability of the payoff functions was used for two arguments.
First, in Theorem~\ref{const-surr} to conclude that the finitistic minmax value of a player in a subgame is independent of the history.
Second, in the proof of Theorem \ref{theorem:1} to ensure that each player~$i$ can play arbitrarily during the first $i$ stages without affecting her overall payoff.
As Voorneveld's \cite{voorneveld2010possibility} example shows, when the payoffs are not tail-measurable,
an $\ep$-equilibrium need not exist for every $\ep > 0$.
It will be interesting to identify other families of payoff functions for which an $\ep$-equilibrium exists for every $\ep > 0$.

\end{document}